\documentclass[12pt,reqno,twoside]{amsart}

\usepackage{amssymb,amsmath,amscd,enumerate,verbatim}
\usepackage[colorlinks=true,linkcolor=blue,citecolor=blue]{hyperref}
\usepackage[latin1]{inputenc}
\usepackage{color}
\usepackage{graphicx}

\usepackage[all]{xy}
\usepackage{pstricks, pst-3d}
\newcommand{\kk}{\mathrm k}

\newcommand{\NN}{\mathbb{N}}

\newcommand{\Ncc}{\mathcal{N}}

\DeclareMathOperator{\pnt}{\raise 0.5mm \hbox{\large\bf.}}

\DeclareMathOperator{\Tor}{Tor}

\DeclareMathOperator{\bigheight}{bigheight}

\DeclareMathOperator{\lcm}{lcm}

\DeclareMathOperator{\reg}{reg}
\DeclareMathOperator{\cd}{cd}

\DeclareMathOperator{\pd}{pd}

\DeclareMathOperator{\supp}{supp}

\let\phi=\varphi
\let\:=\colon

\newtheorem{thm}{\bf Theorem}[section]
\newtheorem{lem}[thm]{\bf Lemma}
\newtheorem{cor}[thm]{\bf Corollary}

\newtheorem{quest}[thm]{\bf Question}

\theoremstyle{definition}
\newtheorem{defn}[thm]{\bf Definition}

\theoremstyle{plain}
\newtheorem*{thm*}{Theorem}
\newtheorem*{lem*}{Lemma}
\newtheorem*{cor*}{Corollary}
\newtheorem*{claim*}{Claim}
\newtheorem*{defn*}{Definition}

\theoremstyle{remark}
\newtheorem{rem}[thm]{Remark}

\numberwithin{equation}{section}
\title{Regularity of linearly presented squarefree monomial ideals}

\author{Hailong Dao}
\address{Department of Mathematics, University of Kansas, Lawrence, KS 66045, USA}
\email{hdao@ku.edu}
\author{Thanh Vu}
\address{Institute of Mathematics, VAST, 18 Hoang Quoc Viet, Hanoi, Vietnam}
\email{vuqthanh@gmail.com}
\thanks{}
\date{\today}

\begin{document}

\begin{abstract} We prove a sharp bound for the regularity of a squarefree monomial ideal with a linear presentation. This result  also answers in positive a question on the cohomological dimension of squarefree monomial ideals satisfying Serre's $S_2$-condition proposed by Dao and Takagi.
\end{abstract}

\maketitle

\section{Introduction} 

In this work, we establish bounds on the Castelnuovo-Mumford regularity of squarefree monomial ideals with a linear presentation. 

\begin{thm}\label{thm1}
Let $I \subseteq S = \kk[x_1,\ldots,x_n]$ be a non-zero squarefree monomial ideal generated in degree $d$. Suppose that $I$ has linear first syzygies. Then 
$$\reg (I) \leq \max \left \{d, \left \lfloor \frac{(d-1)n}{d+1} \right \rfloor + 1 \right \}.$$
\end{thm}

Our bound is sharp for all odd $d>2$, and it was obtained via an elementary, but quite intricate induction process. It is somewhat surprising that such a simple statement is not known until now, despite the intense attention on the topic especially in recent years, which we will describe shortly below. 

To help explain fully the context of our work, it is worth restating our main result via standard equivalences using the Alexander dual as follows:

\begin{cor}\label{cor1}
$I \subseteq S = \kk[x_1,\ldots,x_n]$ be a non-zero squarefree monomial ideal of height $c$. If $S/I$ satisfies Serre condition $(S_2)$, then 
$$ \pd(S/I) = \cd(S,I)\leq \max \left \{c, \left \lfloor \frac{(c-1)n}{c+1} \right \rfloor + 1 \right \}.$$
where $\pd$ denotes projective dimension and $\cd$ denotes the cohomological dimension. 
\end{cor}

Our result was first suggested in the study of the cohomological dimension. Let $I$ be a non-zero ideal in a regular local ring $S$. Recall that $\cd(S,I)$, the cohomological dimension of $I$ is defined as the supremum of $i$ such that $H_I^i(M)\neq 0$ for some $S$-module $M$. Finding bounds on this invariant was proposed by Grothendieck, and such studies have been undertaken by many researchers, see for instance \cite{PS, Ha, HL, Va}. 

When $S$ contains a field $\kk$, Faltings \cite{F} proved that
$$\cd(S,I) \le n - \left \lfloor \frac{n-1}{\bigheight (I) + 1} \right \rfloor,$$
where  $\bigheight(I)$ is the big height of $I$. With no further restriction on $I$, this is the best bound possible. 

Dao and Takagi \cite{DT} studied the cohomological dimension of ideals satisfying Serre's condition $(S_i)$ and proposed the following strengthening of the work of Huneke and Lyubeznik \cite{HL}, who proved such a bound for $I$ such that $S/I$ is normal: 
\begin{quest}\cite [Question 3.10]{DT}\label{questDT} Let $S$ be an excellent regular local ring containing a field and $I\subset S$ be a proper  ideal of height $c$. Assume that $S/I$ satisfies Serre's $(S_2)$ condition. Is it always true that 
$$\cd(S,I) \le n - \left \lfloor \frac{n}{c + 1} \right \rfloor - \left \lfloor \frac{n-1}{c + 1} \right \rfloor?$$
\end{quest}

When $I$ is a squarefree monomial ideal of height $2$, Dao, Huneke, and Schweig \cite{DHS} proved that $\cd(S, I) = O(\log(n))$. Besides that, little is known  even when $I$ is a squarefree monomial ideal. When $I$ is a squarefree monomial ideal, by the results of Terai \cite{T}, Singh and Walther \cite{SW} we have that 
$$\cd (S,I) = \pd(S/I) = \reg (I^\vee),$$
where $\pd(S/I)$ is the projective dimension of $S/I$ and $\reg (I^\vee)$ is the Castelnuovo-Mumford regularity of the Alexander dual of $I$. Furthermore, by \cite[Corollary 3.7]{Y}, the condition that $I$ has height $c$ and satisfies Serre $S_2$-condition is equivalent to the condition that $I^\vee$ is generated in degree $c$ and has linear first syzygies. It is not hard to see that our Corollary \ref{cor1} answers in affirmative  Question \ref{questDT} for squarefree monomial ideals.

To give a broader context for our work from another perspective, we review some literature on ideals with $k-1$ linear steps of resolutions, called $N_{d,k}$ ideals. In algebraic geometry, such conditions when $d=2$ were studied by Green and Lazarsfeld \cite{GL}, Eisenbud, Green, Hulek, and Popescu \cite{EGHP} and others under the name $N_{2,k}$. In combinatorics, the Alexander dual of a square-free $N_{d,k}$ ideal gives a Stanley-Reisner ring satisfying Serre's condition $(S_k)$, whose algebraic and combinatorial properties display remarkable similarities to the Cohen-Macaulay situation, as shown by the work of Murai and Terai \cite{MT} and others. The diameter of the dual graph of such rings has been studied in proposed modifications of the Hirsh conjecture \cite{AB}. Constantinescu, Kahle, and Varbaro \cite{CKV} discovered a stunning connection between such ideals and geometric group theory. They proved that for a Coxeter group $W$, the {\it virtual cohomological dimension} of $W$ is equal to the maximal of the regularity of the Stanley-Reisner ring of its nerve complex $\Ncc(W)$. The defining ideal of this Stanley-Reisner ring is linearly presented and quadratic if and only if $W$ is a right-angled hyperbolic Coxeter group. This connection allows them to build a series of examples of such ideals whose regularity equals  $O(\log(\log(n)))$, answering a question raised in \cite{DHS}.

Recently, Dao and Eisenbud \cite{DE} gave a sharp upper bound for the regularity of $N_{d,k}$ primary monomial ideals. Studying ideals whose regularity achieves the upper bound is an interesting topic. We refer to \cite{DE} for further information.

In the next section, we give the necessary background and prove our main theorem.

\section{Linearly presented squarefree monomial ideals}

Let $S = \kk[x_1,\ldots, x_n]$ be a standard graded polynomial ring over a field $\kk$. We first recall some definitions and results about ideals with partial linear resolutions.

\begin{defn} We say that a homogeneous ideal $I\subseteq S$ satisfies property $N_k$ whenever the minimal free resolution of $I$ over the polynomial ring $S$ is linear for $k-1$ steps. When $I$ satisfies property $N_2$, we also call $I$ linearly presented.
\end{defn}

For a non-zero monomial $f\in S$, the support of $f$, denoted by $\supp (f)$ is the set of all variables $x_i$ such that $x_i$ divides $f$. For the rest of the paper, we assume that $I$ is a monomial ideal of $S$ with the unique monomial minimal generating set $G(I) = \{f_1,\ldots,f_r\}$. The support of $I$ is $\supp (I) = \bigcup_{i=1}^r \supp (f_i)$. For a subset $U \subset [n] = \{1,\ldots,n\}$, we denote by $I(U)$ the restriction of $I$ to $U$, i.e., $I(U) = (f_i \mid \supp (f_i) \subseteq U)$. First, we have

\begin{lem}\label{lem_restriction} Let $I$ be a monomial ideal of $S$. Assume that $I$ satisfies $N_k$. Then for any subset $U\subseteq [n]$, $I(U)$ satisfies $N_k$.
\end{lem}
\begin{proof} The conclusion follows from \cite[Corollary 2.5]{OHH}.
\end{proof}

Now suppose that all generators of $I$ have degree $d$. We define the graph $G_I$ whose vertex set is $G(I)$ and $\{u, v\}$ is an edge of $G(I)$ if and only if $\deg (\lcm(u, v)) = d + 1$. For all $u, v \in G(I)$, let $G_I(u,v)$ be the induced subgraph of $G_I$ with vertex set 
$$V(G_I(u,v)) = \{w  \in G(I) \mid w \text{ divides } \lcm(u, v) \}.$$

The following criterion for a monomial ideal to satisfy property $N_2$ can be deduced from the work of \cite{GPW}. See \cite[Proposition 1.1]{BHZ} or \cite[Proposition 2.2]{DE}.

\begin{lem}\label{lem_N_2_criterion} Let $I$ be a monomial ideal generated in degree $d$. Then $I$ satisfies $N_2$ if and only if $G_I(u,v)$ is connected for all $u,v \in G(I)$.    
\end{lem}

Let $I$ be a monomial ideal and $f$ a monomial of $S$. We denote by $I_f$ and $\bar I_f$ the following monomial ideals 
$$I_f = ( g \in I \mid f \text{ divides } g) \text{ and } \bar I_f = (g/f \mid g \in I_f).$$

\begin{lem}\label{lem_local_N_2} Assume that $I$ satisfies $N_2$ and $f$ is a monomial of $S$. Then $I_f$ and $\bar I_f$ satisfy $N_2$.    
\end{lem}
\begin{proof} We may assume that $f$ divides a minimal generator of $I$ and that $f$ is not a minimal generator of $I$ itself. By Lemma \ref{lem_N_2_criterion}, we deduce that $I_f$ satisfies $N_2$. Now, for any minimal generators $u,v$ of $\bar I_f$, we have $G_{\bar I_f} (u,v)$ is isormorphic to $G_{I_f}(fu,fv)$. The conclusion follows from Lemma \ref{lem_N_2_criterion}.
\end{proof}

Let $I \subseteq S$ be a squarefree monomial ideal and $d \in \NN$ a positive integer. We denote by $I_{[d]}$ the squarefree monomial ideal generated by all squarefree monomials of degree $d$ in $I$.

We have the following key lemma to prove our main result.
\begin{lem}\label{lem_gcd} Let $I$ be a squarefree monomial ideal generated in degree $d$. Let $f$ be a squarefree monomial of $S$ such that $2 \le \deg (f) \le d$, $I \not \subseteq (f)$, and $(I + (f))_{[d]}$ satisfies $N_2$. Note that $f$ might be a minimal generator of $I$. Then there exists a minimal generator $f_1$ of $I$ such that $\deg (\gcd(f_1,f)) = \deg (f) - 1$ and $( I + \gcd(f_1,f))_{[d]}$ satisfies $N_2$.    
\end{lem}
\begin{proof} For simplicity of notation, we denote by $J = (I + (f))_{[d]}$. For ease of reading, we divide the proof into several steps.

\medskip

\noindent{\textbf{Step 1.}} Existence of $f_1$. Since $I \not \subseteq f$, there exists a minimal generator $h$ of $I$ such that $f$ does not divide $h$. Let $g = \gcd (f,h)$ and $f = g f_1$, $h = gh_1$. If $\deg (g) = \deg (f) - 1$ then we are done. Thus, we may assume that $\deg (f) - \deg(g) > 1$. Let $h_2$ be any divisor of $h_1$ such that $\deg (f) + \deg ( h_2) = \deg (h) = d$. Note that if $\deg (f) = d$ then $h_2 = 1$. Since $\gcd(f,h_1)  = 1$, $fh_2$ is a squarefree monomial and $\deg (fh_2) = d$. By assumption $J$ satisfies $N_2$, hence $G_J(fh_2,h)$ is connected. Since $h_2 | h$, $G_{J}(fh_2,h)$ is isomorphic to $G_{\bar J}(f,h/h_2)$ where $\bar J = \bar J_{h_2}$. By Lemma \ref{lem_N_2_criterion} and definition, there exists $u \in G_{\bar J}(f,h/h_2)$ such that $u \in G(\bar J)$ and $\deg (\gcd (u,f)) = \deg (f) - 1$. Let $f_1 = u h_2$, then $f_1 \in I$ and $\deg (\gcd(f,f_1)) = \deg (f) - 1$.

Now, we let $g = \gcd(f,f_1)$ and $K = (I + (g))_{[d]}$. We prove by induction on $d$ and $n$ that $K$ satisfies $N_2$. Let $u,v$ be any two minimal generators of $K$. It suffices to show that there exists a path in $G_K(u,v)$ that connects $u$ and $v$. By Lemma \ref{lem_N_2_criterion}, we may assume that $u \in I$, $u \notin (g)$, $v \in (g)$ and $v \notin I$. 

\medskip

\noindent{\textbf{Step 2.}} Reduction to the case $\gcd(u,g) = 1$. Let $v = gv_1$ and $g_1 = \gcd(u,g)$. If $g_1 \neq 1$, we may consider $\bar J_{g_1}$ and $\bar K_{g_1}$ and by induction on $d$, we deduce that $G_{\bar K_{g_1}} (\bar u, \bar v)$ is connected. Hence, we may assume that $g_1 = 1$. Also, by Step 1, we may write $f = gx$ where $x$ is a variable of $S$.

\medskip

\noindent{\textbf{Step 3.}} Reduction to the case $v_1 | u$. Let $u_1$ be any divisor of $u$ such that $\deg (u_1) = \deg (v_1)$. We have $gu_1  | \lcm (u,v)$. Hence, $gu_1$ is a vertex in $G_K(u,v)$. Clearly, we have a path in $G_K(u,v)$ that connects $gu_1$ and $gv_1$. We may replace $gv_1$ by $g u_1$ if necessary, so we may assume that $v_1 | u$. In other words, we may assume that $u = v_1 h$.

\medskip

\noindent{\textbf{Step 4.}} Reduction to the case where $x$ does not divide $h$. Assume that $x | h$. Let $v_2$ be any divisor of $v_1$ such that $\deg (v_2) = \deg (v_1) -1$. Then, the monomial $gx v_2 \in J$. By assumption, $G_J(u,gxv_2)$ is connected. Note that since $\lcm(u,gxv_2) | \lcm (u,v)$, all vertices of $G_J(u,gxv_2)$ are also vertices of $G_K(u,v)$. Hence, $u$ is connected to $gxv_2$, and $gx v_2$ is connected to $gv_1 = v$ in $G_K(u,v)$. Hence, $u$ is connected to $v$ in $G_K(u,v)$. Thus, we may assume that $x$ does not divide $h$.

\medskip

\noindent{\textbf{Step 5.}} Reduction to the case where $v_1$ is a variable. In recap, now we have $u = v_1 h$, $v = gv_1$ and $f = gx$ where $\supp (v_1), \supp (g), \supp (h)$ and $\{x\}$ are disjoint. Assume that $\deg (v_1) > 1$. We let $v_2$ be any divisor of $v_1$ such that degree $\deg (v_2) = \deg (v_1) - 1$. By Lemma \ref{lem_local_N_2}, $\bar J_{v_2}$ satisfies $N_2$ and $\bar K_{v_2} = (\bar J_{v_2} + (g))_{[d - \deg (v_2)]}$. By induction on $d$, we deduce that $\bar K_{v_2}$ satisfies $N_2$. Hence, $u/v_2$ and $v/v_2$ are connected in $G_{\bar K_{v_2}}(u/v_2,v/v_2)$. Thus, $u$ and $v$ are connected in $G_{K}(u,v)$.

\medskip

\noindent{\textbf{Step 6.}} Conclusion step. By Step 5, we may assume that $\deg (v_1) = 1$. For simplicity, we call $v_1 = y$. Then we have $u = yh$, $v = gy$ and $f = gx$. Also, $\lcm(u,v) = ygh$ and $\lcm(u,f) = xygh$. Let $u_1, \ldots,u_r$ be a shortest path in $G_J(u,f)$ that connects $u$ and $f$. First, assume that there exists $u_i$ such that $x \not | u_i$. Then, we have $u_i | \lcm (u,v)$. We may replace $u$ with $u_i$ and have a shorter path from $u$ to $f$. Note that a path from $u_i$ to $v$ and a path from $u$ to $u_i$ will give us a path from $u$ to $v$ in $G_K(u,v)$. Hence, we may assume that $x | u_i$ for all $i = 1, \ldots, r$. In particular, we must have $u_1 = xh$ as $\gcd(u,u_1) = \deg (u) - 1$. Let $w_i = u_i/x$. Then $yw_1,\ldots,y w_r$ is a path in $G_K(u,v)$ that connects $u$ and $v$. That concludes the proof of the lemma.
\end{proof}

\begin{rem} In general, if $(I + (f))_{[d]}$ satisfies $N_2$ and $g$ divides $f$ then $(I + (g))_{[d]}$ might not satisfy property $N_2$. The condition that $\deg g = \deg f - 1$ is crucial in Lemma \ref{lem_gcd}. For example, Let $S = \kk[x_1,\ldots,x_8]$ and $I = (x_3x_4x_7x_8,x_3x_4x_5x_7,x_3x_5x_6x_7,x_1x_5x_6x_7,x_1x_2x_5x_6)$. Then $I$ has a linear free resolution. Let $f = x_1x_2x_5x_6$ is a minimal generator of $I$ and $g = x_1x_2$, then $ (I + (g))_{[4]}$ does not satisfy property $N_2$.    
\end{rem}

For a finitely-generated graded $S$-module $M$, the (Castelnuovo-Mumford)
regularity of $M$ is defined to be 
$$\reg (M) = \max \{ j - i \mid \Tor_i^S(M,\kk)_j \neq 0 \}.$$
We have the following simple result.

\begin{lem}\label{intersection} Let $J, K$ be non-zero homogeneous ideals of $S$. Then 
\begin{enumerate}
    \item $\reg (J \cap K) \le \max \{ \reg (J), \reg (K), \reg (J + K) + 1\},$
    \item $\reg (J + K) \le \max \{\reg (J), \reg (K), \reg (J \cap K) - 1\}.$
\end{enumerate} 
\end{lem}
\begin{proof}
    We have the short exact sequence 
    $$0 \to J \cap K \to J \oplus K \to (J + K) \to 0.$$
    The conclusion follows from a standard result on regularity along short exact sequences.
\end{proof}

\begin{lem}\label{lem_var_reduction} Let $I$ be a monomial ideal and $g = x_1 \ldots x_k$ be a squarefree monomial such that $g \notin I$. Then 
$$\reg (I,g) \le \max \{\reg (I,x_{i_1}, \ldots, x_{i_j}) + j - 1\}$$ 
where $i_1, \ldots, i_j$ are $j$ distinct indices of $\{1, \ldots, k\}$.    
\end{lem}
\begin{proof}
    We prove by induction on $k$. The case $k = 1$ is obvious. Now assume that $k \ge 2$. Write $(I,g) = (I,g_1) \cap (I,x_k)$ where $g_1 = x_1 \cdots x_{k-1}$. By Lemma \ref{intersection}, we have
    $$\reg (I,g) \le \max (\reg (I,g_1), \reg (I,x_k), \reg (I,g_1,x_k) + 1).$$
    By induction, the conclusion follows.
\end{proof}

Assume that $d \ge 2$. Let $f(n,d)$ and $g(n,d)$ be defined as follows.
$$f(n,d) = \begin{cases} 0 & \text{ if } n < d \\ 
d & \text{ if } n = d\\
\left \lfloor \frac{(d-1)n}{d+1} \right \rfloor + 1 & \text{ if } n > d,\end{cases}$$
and 
$$g(n,d) = n - \left \lfloor \frac{n}{d+1} \right \rfloor - \left \lfloor \frac{n-1}{d+1} \right \rfloor.$$
We now have some simple properties of $f(n,d)$ and $g(n,d)$.
\begin{lem}\label{f_nd_1} Let $n,d,j$ be positive integers. Assume that $d \ge 2$. We have 
\begin{enumerate}
    \item If $j \le \left \lfloor \frac{d+1}{2} \right \rfloor$ and $n - j > d$ then $f(n-j,d) + j-1 \le f(n,d).$
    \item If $j \le d+1$ and $n > d$ then $f(n-j,d) + j - 2 \le f(n,d)$.
\end{enumerate}
\end{lem}
\begin{proof} For (i): by definition, it suffices to prove that 
$$\frac{(d-1)(n-j)}{d+1} + j - 1 \le \frac{(d-1)n}{d+1}.$$
Equivalently, $2j \le d+1$. 

For (ii): by definition, we may assume that $n -j > d$. Then the conclusion is equivalent to $j-2 \le \frac{(d-1)j}{d+1}$, which is equivalent to $j \le d+1$. The conclusion follows.    
\end{proof}
\begin{lem}\label{f_nd_2} Let $n,d$ be positive integers. Assume that $d \ge 2$. Then $g(n,d) - 1 \le f(n,d) \le g(n,d)$ and $f(n,d) = g(n,d) - 1$ if and only if $n = (d+1)k + s$ for some integers $k,s$ such that $1 \le k$ and $\left \lfloor \frac{d+1}{2} \right \rfloor < s \le d.$     
\end{lem}
\begin{proof} By definition, we may assume that $n \ge d+1$. Let $n = (d+1)k + s$ for integers $k$, $s$ such that $0 \le s \le d$. If $s = 0$, then $f(n,d) = g(n,d)$. Thus, we may assume that $1 \le s \le d$. Then $g(n,d) = (d-1)k + s$ and 
$$f(n,d) = (d-1) k + \left \lfloor \frac{s(d-1)}{d+1} \right \rfloor + 1.$$
Since $(s-2) (d+1) < s(d-1) < s(d+1)$, we deduce that $g(n,d) -1 \le f(n,d) \le g(n,d)$. Furthermore, $f(n,d) = g(n,d) - 1$ if and only if $s(d-1) < (s-1)(d+1)$. Equivalently, $d + 1 < 2s$. The conclusion follows.
\end{proof}
\begin{rem} Our bound is slightly better than the bound proposed by Dao and Takagi, as shown in Lemma \ref{f_nd_2}. For example, when $d=5$, the values of $f(n,d)$ and $g(n,d)$ are as below

\begin{center}
    \begin{tabular}{c|c|c|c| c| c | c| c|c | c| c |c| c}
        $n$ & 5 & 6 & 7 & 8 & 9 & 10 & 11 & 12 & 13 & 14 & 15 & $\cdots$\\
        \hline
      $f(n,5)$    & 5 & 5 & 5 & {6} & {7} & \textcolor{blue}{7} & \textcolor{blue}{8} & 9 & {9} & {10} & 11 & $\cdots$\\
      \hline
      $g(n,5)$ &  5 & 5 & 5 & {6} &  {7} & \textcolor{red}{8} & \textcolor{red}{9} & 9 &  {9} &  {10} & 11 & $\cdots$
    \end{tabular}
    
\end{center}
This small distinction is crucial for our induction argument though.
\end{rem}

To prove the main result, we will prove by induction the bound for ideals of the form $I + (f)$ where $f$ is a squarefree monomial of $S$ such that $\deg (f) \le d$ and $ (I  + (f))_{[d]}$ satisfies $N_2$. The initial phase of the induction process has a certain subtlety, and we need to prove it directly in the next two lemmas.

\begin{lem}\label{lem_d+1} Let $I$ be a squarefree monomial ideal generated in degree $d$ and $f$ be a non-zero squarefree monomial in $S = \kk[x_1,\ldots,x_{d+1}]$. Assume that $1 \le \deg (f) \le d$. Then $\reg (I + (f)) \le d$.    
\end{lem}
\begin{proof} We prove by induction on $d$. The base case $d = 1$ is obvious. Now, assume that $d \ge 2$ and $\deg (f) \ge 2$. We may assume that $x_1 | f$. Write $I = x_1 I_1 + I_2$, where $I_2 = (g \mid g \in I, x_1 \text{ does not divide } g)$. Since $\deg (g) = d$ and there are only $d+1$ variables, so either $I_2 = (x_2\cdots x_{d+1})$ or $I_2 = (0)$. When $I_2 = (0)$, the conclusion follows directly from induction. Now, assume that $I_2 = (x_2 \cdots x_{d+1})$. Then we have 
$$ I + (f) = x_1 ( I_1 + (f_1)) + I_2.$$
Furthermore, $I_2 \subseteq I_1 + (f_1)$, hence $x_1 (I_1 + (f_1)) \cap I_2 = x_1 I_2$. The conclusion follows from Lemma \ref{intersection} and induction.    
\end{proof}
\begin{cor} Any ideal generated by squarefree monomials of degrees $d$ in $d+1$ variables has a linear free resolution.    
\end{cor}
\begin{proof}
    The conclusion follows immediately from Lemma \ref{lem_d+1} by setting $f$ as one of the minimal generators of the ideal itself.
\end{proof}

\begin{lem}\label{lem_d+2} Let $I$ be a squarefree monomial ideal generated in degree $d$ and $f$ be a non-zero squarefree monomial in $S = \kk[x_1,\ldots,x_{d+2}]$. Assume that $1 \le \deg (f) \le d$ and $I$ and $(I + (f))_{[d]}$ satisfy $N_2$. Then $\reg (I + (f)) \le d$.    
\end{lem}
\begin{proof} We prove by induction on $d$. The base case $d = 1$ is obvious. When $\deg (f) = 1$, the conclusion follows from induction and Lemma \ref{lem_restriction}. Now, assume that $d \ge 2$ and $\deg (f) \ge 2$. We may assume that $x_1 | f$. Write $f = x_1 f_1$ and $I = x_1 I_1 + I_2$, where $I_2 = (g \mid g \in I, x_1 \text{ does not divide } g)$. We may assume that $I_2$ is non-zero. Let 
$$J_2 = (g \mid g \text{ is a minimal generator of } I_2 \text{ and } g \subseteq I_1 + (f_1)).$$
We will prove that $x_1 (I_1 + (f_1)) \cap I_2 = x_1 J_2.$ Indeed, let $h = x_1 h_1$ be a minimal generator of $x_1 I_1$ and $g$ be a minimal generator of $I_2$. Since $x_1 \notin \supp (g)$, we may assume that $g = x_2\cdots x_{d+1}$. If $\lcm(h,g) = x_1g$, we are done. Thus, we may assume that $\lcm(h,g) = x_1\cdots x_{d+2}$. In other words, every minimal generator of $I$ will be a vertex in $G_I(h,g)$. Let $h = u_0,u_1, \ldots, u_r,g = u_{r+1}$ be a path in $G_I(h,g)$. Let $i$ be the smallest index such that $x_1$ does not divide $u_i$. Then $\lcm(u_{i-1}, u_i) = x_1 u_i$. Since $u_{i-1} \in x_1I_1$ and $u_i \in I_2$, we deduce that $x_1 u_i \in x_1 I_1 \cap I_2$. Hence, $\lcm(h,g)$ is not minimal in $x_1 I_1 \cap I_2$. The intersection $(f_1) \cap I_2$ can be done similarly. 

Now, by Lemma \ref{intersection}, we have 
$$\reg ( I +(f)) \le \max \{ \reg (x_1 (I_1 + (f_1))), \reg (I_2), \reg (x_1 J_2) - 1\}.$$
By induction and Lemma \ref{lem_d+1}, the conclusion follows.
\end{proof}

\begin{rem}
Alternatively, one can prove Lemma \ref{lem_d+2} as follows. It suffices to prove that the square-free part of $(I + (f))_{[d]}$ has linear resolution, since if $\reg (I + (f)) > d$, there has to be a nontrivial homology $\Tor_i(k, S/(I,f))_m$ at square-free monomial degree $m$ with $|m|>i+d$, as $I+(f)$  is a square-free monomial ideal. But the square-free part of  $(I + (f))_{[d]}$ still satisfies $N_2$, so it's Alexander dual $J$ is $(S_2)$. However $J$ has codimension $d$ in a ring of $d+2$ variable, so $S/J$ is Cohen-Macaulay, and thus the square-free part of $(I + (f))_{[d]}$ has linear resolution. 
\end{rem}

We are now ready for the proof of the main theorem.
\begin{proof}[Proof of Theorem \ref{thm1}] We prove by induction on $n$ that $\reg (I + (f)) \le f(n,d)$ for all squarefree monomial ideal $I$ and all squarefree monomial $f$ of $S = \kk[x_1,\ldots,x_n]$ such that $I$ is generated in degree $d$, and $I$ and $(I + (f))_{[d]}$ satisfies $N_2$. By Lemma \ref{lem_d+1} and Lemma \ref{lem_d+2}, we may assume that $n > d+2$. For ease of reading, we divide the proof into several steps.

\medskip

\noindent{\textbf{Step 1.}} $d + 2 < n \le d + \left \lfloor \frac{d+1}{2} \right \rfloor + 1$. We prove by induction on the degree of $f$. Note that, in this range, we have $f(n,d) = f(n-1,d) + 1$. If $\deg(f) = 1$, the conclusion follows from Lemma \ref{lem_restriction}. Now, assume that $\deg(f) \ge 2$. By Lemma \ref{lem_gcd}, there exists $g$ such that $f = xg$ for some variable $x$ of $S$ and $(I + (g))_{[d]}$ satisfies $N_2$. We have 
$$ (I + (f)) = (I + (g)) \cap (I + (x)).$$ 
By Lemma \ref{intersection}, we deduce that 
$$\reg ( I +(f)) \le \max \{ \reg ( I + (g)), \reg ( I + (x)), \reg ( I +(g,x)) + 1\}.$$
By induction on the degree of $f$, $\reg (I + (g)) \le f(n,d)$. By induction on $n$, $\reg (I +(x))$ and $\reg (I + (g,x)) \le f(n-1,d)$. Since $f(n,d) = f(n-1,d)+1$, the conclusion follows.

\medskip
We now assume that $n >  d + \left \lfloor \frac{d+1}{2} \right \rfloor + 1$.

\medskip

\noindent{\textbf{Step 2.}}  $\deg (f) \le \left \lfloor \frac{d+1}{2} \right \rfloor$. By Lemma \ref{lem_var_reduction}, we have 

\begin{align*}
    \reg ( I + (f)) &\le \max \{ \reg (I + (x_{i_1}, \ldots, x_{i_j})) + j-1 \mid \emptyset \neq \{i_1,\ldots,i_{j}\} \subseteq \supp (f)\} \\
    &\le \max \{ f(n-j,d) + j-1\mid j = 1, \ldots, \deg (f)\} \le f(n,d).
\end{align*}
The second inequality follows from induction and Lemma \ref{lem_restriction}, and the third inequality follows from Lemma \ref{f_nd_1}.

\medskip

\noindent{\textbf{Step 3.}} $\deg (f) > \left \lfloor \frac{d+1}{2} \right \rfloor$. By Lemma \ref{lem_gcd}, there exists $f_1 \in I$ such that $\gcd (f_1,f) = g$, $\deg (g) = \deg (f) - 1$ and $(I + (g))_{[d]}$ satisfies $N_2$. Write $f = xg$ and $f_1 = g f_2$. We have $\deg (f_2) \le \left \lfloor \frac{d+1}{2} \right \rfloor$ and 
$$I + (f) = (I + (g)) \cap (I + (f_2,x).$$
By Lemma \ref{intersection}, we deduce that 
$$\reg (I + (f)) \le \max \{ \reg (I + (g)), \reg ( I +(f_2,x)), \reg (I + (g,f_2,x)) + 1\}.$$
By induction on the degree of $f$, we have $\reg (I  + (g)) \le f(n,d)$. Since, $\deg (f_2) \le \left \lfloor \frac{d+1}{2} \right \rfloor$, as in Step 1, we deduce that $\reg (I + (f_2,x)) \le f(n,d)$. Note that $\supp (g)$, $\supp (f_2)$ and $\{x\}$ are distinct and $\deg (g) + \deg (f_2) = d$. Applying Lemma \ref{lem_var_reduction} twice to $g$ and $f_2$, we deduce that 
\begin{align*}
    \reg (I + (f_2,g,x)) + 1 & \le \max \{ \reg ( I + (x,x_{i_1},\ldots,x_{i_s},x_{j_1},\ldots,x_{j_t})) + (s + t - 1)  \\
     & \mid \{i_1,\ldots,i_s\} \subseteq \supp (f_2), \{j_1,\ldots,j_t\} \subseteq \supp (g)\}.
\end{align*}
By induction, Lemma \ref{lem_restriction} and Lemma \ref{f_nd_1}, we deduce that 
$$\reg (I + (g,f_2,x)) + 1 \le \max \{ f(n-j-1,d) + j-1 \mid j = 1, \ldots, \deg(f_2) + \deg (g)\} \le f(n,d).$$
The conclusion follows.
\end{proof}
Dao and Takagi \cite[Example 3.11]{DT} showed that the bound in Theorem \ref{thm1} is achieved when $d = 2k+1$ and $n = t(k+1)$ for $t \ge 2$. We will now show that the bound in \ref{thm1} is achieved for all $n \ge d$ when $d$ is odd.
\begin{lem}\label{lem_sharp_bound} Assume that $d = 2k + 1 \ge 3$. Then for all $n \ge d$, there exists an ideal $I$ generated in degree $d$ such that $I$ satisfies property $N_2$ and $\reg (I) = f(n,d)$.    
\end{lem}
\begin{proof} We may assume that $n \ge d + 1$. Let $n = (k+1) t + s$ for $0\le s \le k$ and $t \ge 2$. We have 
$$f(n,d) = kt + \left \lfloor \frac{ k s}{k+1} \right \rfloor + 1 = \begin{cases}
    kt + 1 & \text{ if } s = 0, \\
    kt + s & \text{ if } s \ge 1.
\end{cases}$$
If $s = 0$ or $s = 1$, we let $J$ be the complete intersection of $t$ monomials of degree $k+1$. If $2 \le s \le k$, we let $J$ be the complete intersection of $t$ monomials of degree $k+1$ and a monomial of degree $s$. Then, we have $\reg (J) = f(n,d)$. Let $I = J_{[d]}$. By Lemma \ref{lem_N_2_criterion}, \cite[Theorem 2.1]{GPW}, and the truncation principle \cite[Proposition 1.7]{EHU}, we have that $I$ satisfies property $N_2$ and $\reg (I) = \reg (J) = f(n,d)$. The conclusion follows.
\end{proof}

\begin{cor} Let $I$ be a squarefree monomial ideal of height $c$. Assume that $S/I$ satisfies Serre's $S_2$ condition. Then 
$$\cd(S,I) \le f(n,c) \le n - \left \lfloor \frac{n}{c + 1} \right \rfloor - \left \lfloor \frac{n-1}{c + 1} \right \rfloor.$$
\end{cor}
\begin{proof}
    The conclusion follows from the result of Terai \cite{T}, Singh and Walther \cite{SW}, Yanagawa \cite{Y}, Lemma \ref{f_nd_2}, and Theorem \ref{thm1}.
\end{proof}

\begin{rem} When $d = 2\ell$ is even, the bound in Theorem \ref{thm1} might not be tight yet. The first case, where we are not able to construct a tight example, is when $d = 4$ and $n = 10$. 
\end{rem}

\section*{Acknowledgments} A part of this paper was completed while the first author visited the Vietnam Institute for Advanced Study in Mathematics (VIASM) in June 2024. He would like to thank the VIASM for hospitality and financial support.

\end{document}